\documentclass{article}
\usepackage{graphicx}
\usepackage[english]{babel}
\usepackage{amsmath}
\usepackage{amsthm}
\usepackage{listings}
\usepackage{hyperref}
\usepackage{tikz}
\usepackage{float}
\usepackage{authblk}

\hypersetup{
    colorlinks=true,
    linkcolor=blue,
    filecolor=magenta,      
    urlcolor=cyan
    }

\newtheorem{theorem}{Theorem}[section]
\theoremstyle{definition}
\newtheorem{definition}{Definition}[section]
\newtheorem{corollary}{Corollary}[theorem]
\newtheorem{lemma}[theorem]{Lemma}
\newtheorem*{remark}{Remark}

\title{Enumeration and Distribution of Permutation Rows and Columns in Equi-$n$-Squares}
\author[1]{Andrew Pendleton\thanks{ajpendleton.math@gmail.com}}
\affil[1]{Department of Mathematical Sciences, Rutgers University--Camden, Camden, NJ, United States}
\date{January 2026}

\begin{document}

\maketitle

\begin{abstract}
    We introduce consecutive equi-$n$-squares, a variant of equi-$n$-squares in which at least one row or column forms a fixed permutation of $\{1,\dots,n\}$, taken for concreteness to be $(1,\dots,n)$. More generally, the enumeration and probabilistic arguments presented here extend to the occurrence of any prescribed permutation as a row or column of an equi-$n$-square. We derive exact and asymptotic formulas for the number of consecutive equi-$n$-squares, showing precisely how their proportion among all equi-$n$-squares rapidly approaches zero as $n\to\infty$.
    
    We also analyze the distribution of consecutive equi-$n$-squares under uniform random sampling and explore connections to algebraic structures, interpreting equi-$n$-squares and consecutive equi-$n$-squares as Cayley tables. Finally, we supplement our theoretical results with Monte Carlo simulations for small values of $n$.
\end{abstract}

\section{Introduction} \label{intro}

\begin{definition}[Latin Square]
    A \textit{Latin square} of order $n$ is an $n$ by $n$ array with entries from the multiset \[\{1,1,\dots, 1,2,2,\dots ,n-1,n,n,\dots, n\}\] where each element has a multiplicity of $n$ and each element appears exactly once in each row and column.
\end{definition}

For example, consider this Latin square of order 5:
\[
\begin{array}{|c|c|c|c|c|}\hline
    1 & 2 & 3 & 4 & 5\\\hline
    5 & 1 & 2 & 3 & 4\\\hline
    4 & 5 & 1 & 2 & 3\\\hline
    3 & 4 & 5 & 1 & 2\\\hline
    2 & 3 & 4 & 5 & 1\\\hline
\end{array}.
\]

The study of Latin squares is ancient. West and South Asian numerologists and mathematicians---perhaps as early as the year 1000---took an interest in their supposed mystical properties and constructed many \cite{ComboAncMod}. The study of Latin squares in the Western mathematical tradition was initiated by Euler in his attempts to better understand the construction of magic squares and solve the thirty-six officers problem. 

Latin squares have since become objects of study apart from their connections to magic squares and Euler's officers problem. The study of Latin squares has since been connected to the study of quasigroups (see section \ref{algebraic interpretations}) \cite{SmallLSsQGsLs}, experimental design in statistics \cite{Bailey1996}, error correction \cite{10.5555/3179439.3179443, AHMED2022104676}, and cryptography \cite{SecretSharingSchemes, LSCryptoApps}.

Our study of Latin squares is inspired by the website \href{https://donotfindthefox.com/classic}{donotfindthefox.com}. This website hosts a game in which the player attempts to randomly place 5 ``F''s, 6 ``O''s, and 5 ``X''s on a 4 by 4 grid without spelling the word ``FOX'' in any direction---whether forwards or backwards---along any column, row, or diagonal. Of course, this is a game of complete chance, and one which is not that likely to be won in the player's first few attempts. 

The combinatorial flavor of this game is self-evident, but determining one's probability of success resists a very satisfying, simple solution. Furthermore, we will see that small changes in the order of the squares we study massively impact the distribution of their features. As with Euler's officers problem, Latin squares and their variants often produce surprising and counterintuitive results. This paper seeks to develop the theory of objects similar to the board in the ``FOXs in Boxes'' game. The classification we arrive at is rather thorough and may provide a window into the study of related objects (e.g. magic squares, MOLS, etc.). 

\begin{definition}[Equi-$n$-Square]
     An \textit{equi-$n$-square} is an $n$ by $n$ array with entries from the multiset \[\{1,1,\dots, 1,2,2,\dots ,n-1,n,n,\dots, n\}\] where each element has a multiplicity of $n$.
\end{definition} 

For example, consider this equi-$3$-square:
\[
\begin{array}{|c|c|c|}\hline
    1 & 2 & 1 \\\hline
    3 & 3 & 1 \\\hline
    3 & 2 & 2 \\\hline
\end{array}.
\]

Equi-$n$-squares are objects of active research. Most famously, Stein conjectured that in every equi-$n$-square there exists a transversal of size $n-1$ \cite{MR1130611,MR387083}. Recently, this conjecture has been disproven by Pokrovskiy and Sudakov \cite{Pokrovskiy2019-vb}. Furthermore, Aharoni, Berger, Koltar, and Ziv have shown the size of the largest transversal to be at least $2n/3$ \cite{Aharoni2017} and Chakraborti, Christoph, Hunter, Montgomery, and Petrov have shown that any equi-$n$-square has a transveral with size $(1-o(1))n$ \cite{Chakraborti2024-pc}.

Apart from Stein's Conjecture and related problems, connections have been found between the study of equi-$n$-squares and graph theory \cite{Akbari2020-fd}, computational complexity \cite{Vaughan2009-om}, and Sudoku \cite{Bailey01052008}.

\begin{definition}[Consecutive Equi-$n$-Square] \label{CPLS}
    An equi-$n$-square is called \textit{consecutive} if it has at least one consecutive or reverse-consecutive row or column. 
\end{definition}

For example, consider this consecutive equi-$4$-square with one consecutive row and two consecutive columns:
\[
\begin{array}{|c|c|c|c|}\hline
    1 & 2 & 3 & 4 \\\hline
    2 & 4 & 1 & 3 \\\hline
    3 & 3 & 2 & 2 \\\hline
    4 & 4 & 1 & 1 \\\hline
\end{array}.
\]

Henceforth, we will abbreviate ``equi-$n$-square'' as E$n$S and ``consecutive equi-$n$-square'' as CE$n$S. Let $\Omega_n$ be the set of all E$n$Ss and $\Sigma_n$ the set of all CE$n$Ss. We seek to determine the probability that a randomly chosen E$n$S, say $\omega\in\Omega_n$ is consecutive, that is to calculate $P(\omega\in\Sigma_n).$

It is important to note that none of the arguments here rely on choosing the consecutive/increasing permutation $(1,\dots,n)$. More generally, one may extend these arguments verbatim to enumerate the occurrence of any given permutation of a set of $n$ distinct elements in an equi-$n$-square. The framing of consecutive rows and columns composed of the symbols $1,\dots, n$ is intended to highlight the continuity of CE$n$Ss as a natural generalization of equi-$n$-squares and (in particular, reduced) Latin squares.

\begin{theorem} \label{main_theorem}
    The number of consecutive equi-$n$-squares (the number of elements in  $\Sigma_n$) is \[
    \begin{split}
        \left(\sum_{i=1}^{n}(-2)^{i+1}\binom{n}{i}\dfrac{(n^2-in)!}{(n-i)!^{n}}\right)-\dfrac{(4n)(n^2-2n+1)!}{(n-1)!(n-2)!^{n-1}}\\
        \qquad+\dfrac{8\lfloor n/2\rfloor(n^2-3n+2)!}{(n-2)!^2(n-3)!^{n-2}}-\dfrac{2\lfloor n/2\rfloor(n^2-4n+4)!}{(n-2)!^2(n-4)!^{n-2}}
    \end{split}
    \]
    and
    \[
    |\Sigma_n|\sim\frac{4n\left(n^{2}-n\right)!}{\left(n-1\right)!^{n}}.
    \]
\end{theorem}

\begin{corollary} \label{asymp_coro}
    Let $\omega\sim\mathcal{U}(\Omega_n)$ be a random E$n$S. The probability that $\omega$ is also a CE$n$S is \[
    P(\omega\in\Sigma_n)\sim\dfrac{4n^{n+1}(n^2-n)!}{n^2!}.
    \]
\end{corollary}

\begin{theorem} \label{distro_theorem}
    Let $\omega\sim\mathcal{U}(\Omega_n)$ be a random E$n$S. Define a random variable $X_n:\Omega_n\to\mathbf{R}$ to count the number of consecutive rows or columns in $\omega$. Then, the probability mass function $p_{X_n}(x)$ of $X_n$ is
    
    {\footnotesize
    \[p_{X_n}(x)=
    \begin{cases}
       1-\frac{n!^n}{n^2!}\Bigg(\left(\sum_{i=1}^{n}(-2)^{i+1}\binom{n}{i}\frac{(n^2-in)!}{(n-i)!^{n}}\right)-\frac{(4n)(n^2-2n+1)!}{(n-1)!(n-2)!^{n-1}}\\
       \qquad+\frac{8\lfloor n/2\rfloor(n^2-3n+2)!}{(n-2)!^2(n-3)!^{n-2}}-\frac{2\lfloor n/2\rfloor(n^2-4n+4)!}{(n-2)!^2(n-4)!^{n-2}}\Bigg)\Bigg)&\text{if}~x=0\\
       
       \frac{n!^n}{n^2!}\Bigg(\left(n\sum_{i=1}^{n}(-2)^{i+1}\binom{n-1}{i-1}\frac{(n^2-in)!}{(n-i)!^n}\right)-\frac{(8n)(n^2-2n+1)!}{(n-1)!(n-2)!^{n-1}}\\
       +\frac{8(\mathbf{1}_{2\mathbf{N}}(n)+n+\lfloor n/2\rfloor-1)(n^2-3n+2)!}{(n-2)!^2(n-3)!^{n-2}}-\frac{(9\cdot\mathbf{1}_{2\mathbf{N}}(n)+9n+2\lfloor n/2\rfloor-9)(n^2-4n+4)!}{(n-2)!^2(n-4)!^{n-2}}\Bigg)&\text{if}~x=1\\
       
       \frac{n!^n}{n^2!}\Bigg(\frac{4n(n^2-2n+1)!}{(n-1)!(n-2)!^{n-1}}+2^3\binom{n}{2}\left(\sum_{k=0}^{n-2}(-2)^k\binom{n-2}{k}\frac{(n^2-2n-nk)!}{(n-k-2)!^n}\right)\\
       \qquad-3\bigg(\frac{4(n+\mathbf{1}_{2\mathbf{N}}(n)-1)(n^2-3n+2)!}{(n-2)!^2(n-3)!^{n-2}}-\frac{4(n+\mathbf{1}_{2\mathbf{N}}(n)-1)(n^2-4n+4)!}{(n-2)!^2(n-4)!^{n-2}}\bigg)\Bigg)&\text{if}~x=2\\
       
       \frac{n!^n}{n^2!}\Bigg(2^4\binom{n}{3}\left(\sum_{k=0}^{n-3}(-2)^k\binom{n-3}{k}\frac{(n^2-3n-nk)!}{(n-k-3)!^n}\right)+\frac{4(n+\mathbf{1}_{2\mathbf{N}}(n)-1)(n^2-3n+2)!}{(n-2)!^2(n-3)!^{n-2}}\\
       \qquad-\frac{4(n+\mathbf{1}_{2\mathbf{N}}(n)-1)(n^2-4n+4)!}{(n-2)!^2(n-4)!^{n-2}}\Bigg)&\text{if}~x=3\\
       
       \frac{n!^n}{n^2!}\Bigg(2^5\binom{n}{4}\left(\sum_{k=0}^{n-4}(-2)^k\binom{n-4}{k}\frac{(n^2-4n-nk)!}{(n-k-4)!^n}\right)+\frac{(n+\mathbf{1}_{2\mathbf{N}}(n)-1)(n^2-4n+4)!}{(n-2)!^2(n-4)!^{n-2}}\Bigg)&\text{if}~x=4\\
       
       \frac{n!^n}{n^2!}\Bigg(2^{x+1}\binom{n}{x}\sum_{k=0}^{n-x}(-2)^k\binom{n-x}{k}\frac{(n^2-nx-nk)!}{(n-x-k)!^n}\Bigg)&\text{if}~x>4
    \end{cases}.
    \]}
\end{theorem}

This Theorem \ref{distro_theorem} demonstrates the dramatic scarcity of CE$n$Ss among higher-order E$n$Ss, which may be of interest to those studying the cryptographic applications of Latin squares.

In section \ref{algebraic interpretations}, we discuss the algebraic implications of these results, proving 

\begin{theorem} \label{no QGs are Ls}
    As $n\to \infty$, the proportion of quasigroups among magmas with $n$ elements is less than $n/n!$.
\end{theorem}

\begin{definition}[\cite{HowieSemigroups} Completely simple semigroup]
    A semigroup $S$ is \textit{simple} if it has no ideals other than itself. A simple semigroup $S'$ is \textit{completely simple} if there exists some idempotent $e\in S$ such that $e$ is primitive for every other idempotent; that is for every other idempotent $f$, \[ef=fe=f\neq 0\implies e=f.\]
\end{definition}

\begin{theorem} \label{completely simple}
     The Cayley table of every finite completely simple semigroup of order $n$ is an equi-$n$-square.
\end{theorem}

\begin{corollary} \label{nonLatin_coro}
    Not all equi-$n$-squares that are associative Cayley tables are Latin squares.
\end{corollary}

Finally, in section \ref{code} we detail the computational methods which confirm the above results.

\begin{remark} \label{notation}
(Regarding Notation)

When referring to a row or column of a consecutive equi-$n$-square, take ``consecutive'' to mean ``consecutive or reverse-consecutive,'' unless otherwise specified. It should be noted that the choice between a consecutive or reverse-consecutive row or column is binary. In many of the following formulas which result from constructive arguments, there is an additional factor of 2 which comes about from this choice---even if it is not stated explicitly in said arguments.

Furthermore, we define the opposite of a row or column $i$ to be row or column $i':=n-i+1$, respectively. We will also make frequent use of the indicator function on some set $A$, $\mathbf{1}_A$. As is convention, we define $\mathbf{1}_A$ in the following way,
\[
\mathbf{1}_A(a):=
\begin{cases}
    1&\text{if}\;a\in A\\
    0&\text{otherwise}
\end{cases}.
\]
\end{remark}

\section{Lemmas \label{lemmas}}

\begin{lemma} \label{O}
    The number of equi-$n$-squares is 
    \[
    |\Omega_n|=\dfrac{n^2!}{n!^n}.
    \]
\end{lemma}

\begin{proof}
    Every equi-$n$-square is an arrangement of the multiset containing $n$ copies of each symbol $1,\dots,n$. Hence the number of such arrays is the multinomial coefficient
    \[
    |\Omega_n|=\frac{(n^2)!}{(n!)^n,
    }
    \]
    which equals the expression stated.
\end{proof}

\begin{lemma} \label{R}
    Let $R$ be the set of all E$n$Ss with at least one consecutive row. The number of such E$n$Ss is
    \[
    |R|=\sum_{i=1}^{n}(-1)^{i+1}\cdot 2^i\cdot\dfrac{n!(n^2-in)!}{(n-i)!^{n+1}i!}.
    \]
\end{lemma}

\begin{proof}
    Let $R_i\subseteq R$ be the set E$n$Ss with consecutive row $i\in\{1,\dots,n\}$. By our definition of $R$, we have \[R=\bigcup_{i}R_i.\] To determine the size of $R$, we construct an arbitrary element. First let $i$ rows be consecutive. Then, the remaining elements may be placed. Finally we sum over all $i$ between $1$ and $n$, employing PIE to avoid any double counting.
\[
\begin{split}
    |R| & =\sum_{i=1}^{n}(-1)^{i+1}\binom{n}{i}2^i\prod_{j=0}^{n-1}\binom{n^2-in-j(n-i)}{n-i}\\
    & =\sum_{i=1}^{n}(-1)^{i+1}\binom{n}{i}2^i\cdot\dfrac{(n^2-in)!}{(n-i)!^n}.\\
\end{split}
\]
\end{proof}

\begin{lemma} \label{M}
    Let $R_{\frac{n+1}{2}}\cap C$ be the set of E$n$Ss with a consecutive middle row (should it exist) and at least one consecutive column. If we define
    \[M:=|R_{\frac{n+1}{2}}\cap C|, \] then
    \[M=\dfrac{4\cdot\mathbf{1}_{2\mathbf{N}}(n+1)(n^2-2n+1)!}{(n-1)!(n-2)!^{n-1}},\] where $\mathbf{1}_{2\mathbf{N}}$ is the indicator function of the even natural numbers.
\end{lemma}

\begin{proof}
    If $n$ is even then $M=0$. If $n$ is odd, we determine $M$ by deciding the order of the middle row and the middle column, which is the only choice for a consecutive column. Note that the number $\frac{n+1}{2}$ will have only occurred once so far in the grid, so there remain $n-1$ copies to be placed. Once this is done, the locations of the remaining $n-2$ copies of the remaining $n-1$ elements are placed. Using this procedure,
\[
\begin{split}
    M & =4\cdot\mathbf{1}_{2\mathbf{N}}(n+1)\binom{n^2-2n+1}{n-1}\prod_{k=0}^{n-2}\binom{n^2-2n+1-(n-1)-k(n-2)}{n-2}\\
    & = \dfrac{4\cdot\mathbf{1}_{2\mathbf{N}}(n+1)(n^2-2n+1)!}{(n-1)!(n-2)!^{n-1}}.
\end{split} 
\]
\end{proof}

\begin{lemma} \label{RC}
    Let $R_i\cap C_i$ be the set of E$n$Ss with consecutive row $i$ and consecutive column $i$. The number of such E$n$Ss is
    \[
    |R_i\cap C_i|=\dfrac{2(n^2-2n+1)!}{(n-1)!(n-2)!^{n-1}}.
    \]
\end{lemma}

\begin{proof}
     First, we determine the order of row $i$, which determines the order of column $i$. Note now that the element $i$ has only occurred once in the array, whereas all others have occurred twice. We place the remaining $n-1$ copies of element $i$, then the $n-2$ copies of all of the other elements. Using this procedure,
    \[
    \begin{split}
    |R_i\cap C_i| & =2\binom{n^2-2n+1}{n-1}\prod_{k=0}^{n-2}\binom{n^2-2n+1-(n-1)-k(n-2)}{n-2}\\
    & = \dfrac{2(n^2-2n+1)!}{(n-1)!(n-2)!^{n-1}}.
    \end{split}
    \]
    
\end{proof}

\begin{lemma} \label{RCC}
    Let $R_i\cap C_i\cap C_{i'}$ be the set of E$n$Ss with consecutive row $i$, consecutive column $i$, and consecutive column $i'$. The number of such E$n$Ss is
    \[
    |R_i\cap C_i \cap C_{i'}|=\dfrac{2(n^2-3n+2)!}{(n-2)!^2(n-3)!^{n-2}}.
    \]
\end{lemma}

\begin{proof}
    By similar argument to Lemma \ref{RC}, \[
    \begin{split}
    |R_i\cap C_i\cap C_{i'}| & =2\binom{n^2-3n+2}{n-2}\binom{n^2-3n+2-(n-2)}{n-2}\\
    & \qquad \cdot\prod_{k=0}^{n-3}\binom{n^2-3n+2-2(n-2)-k(n-3)}{n-3}\\
    & = \dfrac{2(n^2-3n+2)!}{(n-2)!^2(n-3)!^{n-2}}
    \end{split}
    \]
\end{proof}

\begin{lemma} \label{RRCC}
    Let $R_i\cap R_{i'}\cap C_{i}\cap C_{i'}$ be the set E$n$Ss with consecutive row $i$, consecutive row $i'$, consecutive column $i$, and consecutive column $i'$. The number of such E$n$Ss is
    \[
    |R_i\cap R_{i'}\cap C_i \cap C_{i'}|=\dfrac{2(n^2-4n+4)!}{(n-2)!^2(n-4)!^{n-2}}.
    \]
\end{lemma}

\begin{proof}
    By similar argument to Lemma \ref{RC}, \[
    \begin{split}
    |R_i\cap R_{i'}\cap C_i\cap C_{i'}| & =2\binom{n^2-4n+4}{n-2}\binom{n^2-4n+4-(n-2)}{n-2}\\
    &\qquad\cdot\prod_{k=0}^{n-3}\binom{n^2-4n+4-2(n-2)-k(n-4)}{n-4}\\
    & = \dfrac{2(n^2-4n+4)!}{(n-2)!^2(n-4)!^{n-2}}.
    \end{split}
    \]
\end{proof}

\section{Proof of Theorem \ref{main_theorem}}
\begin{proof} \label{proof}

We proceed by repeated application of Principle of Inclusion and Exclusion. Let $R$ be the set of all E$n$Ss with at least one consecutive row. Let $C$ be the set of all E$n$Ss with at least one consecutive column. From Definition \ref{CPLS} and PIE we have that $\Sigma_n=R\cup C$ and 
\[
|\Sigma_n|=|R|+|C|-|R\cap C|.
\]
Noticing that a rotation of $90^\circ$ is a bijection from $R$ onto $C$, we then have 
\begin{equation} \label{pie1}
    |\Sigma_n|=2|R|-|R\cap C|.
\end{equation}

Some subtlety is required in finding the size of the intersection $R\cap C$. Again, we employ the fact that $R=\bigcup_{i}R_i$.
\[\begin{split}
    |R\cap C| & =\left|\left(\bigcup_iR_i\right)\cap C\right|\\
    &=|R_1\cap C|+|R_2\cap C|+\dots+|R_n\cap C|\\& 
    \quad-|R_1\cap R_2\cap C|-|R_1\cap R_2\cap C|-\dots-|R_n\cap R_n\cap C|\\
    &\quad+\dots.
\end{split}\] 
Note that we can safely ignore the higher order terms involving more than 3 intersections as---by the Pigeonhole Principle---if there are 3 or more consecutive rows, there cannot be a consecutive column. Further note that $R_i\cap R_j\cap C=\emptyset$ for $j\neq i'$, so
\begin{equation}\ \label{pie2}
    |R\cap C|=\left(\sum_{i=1}^n|R_i\cap C|\right)-\left(\sum_{i=1}^{\lfloor n/2\rfloor}|R_i\cap R_{i'}\cap C|\right).
\end{equation}

Similarly, a consecutive row $i$ precludes the existence of a consecutive column except in column $i$ or $i'$. Again by PIE, noting that $C=\bigcup_i C_i$, we have 
\[
|R_i\cap C|=|R_i\cap C_i|+|R_i\cap C_{i'}|-|R_i\cap C_i\cap C_{i'}|.
\]
Note that the map from $R_i\cap C_i$ to $R_i\cap C_{i'}$ defined by a vertical reflection and reversing the relevant row and column constitutes a bijection. Therefore the above equation can be further simplified to 
\[
|R_i\cap C|=2|R_i\cap C_i|-|R_i\cap C_i\cap C_{i'}|.
\]
Again, by the map from $R_i\cap C$ to $R_{i'}\cap C$ defined by a reflection through the horizontal axis of the array, we can rewrite $\sum_{i=1}^n|R_i\cap C|$ as
\begin{equation} \label{pie3}
\sum_{i=1}^{n}|R_i\cap C|=M+2\sum_{i=1}^{\lfloor n/2\rfloor} 2|R_i\cap C_i|-|R_i\cap C_i\cap C_{i'}|.
\end{equation}
Where $M$, is defined as in Lemma \ref{M}.

All that remains to be done is to find the size of $R_i\cap R_{i'}\cap C$. Again by PIE, we have \begin{equation} \label{pie4}
    |R_i\cap R_{i'}\cap C|=2|R_i\cap R_{i'}\cap C_i|-|R_i\cap R_{i'}\cap C_i\cap C_{i'}|.
\end{equation}
Finally, note that the map from $R_i\cap C_i\cap C_{i'}$ to $R_i\cap R_{i'}\cap C_i$ defined by a $90^{\circ}$ rotation is a bijection.

From equations \eqref{pie1} to \eqref{pie2}, we have 
\[
\begin{split}
    |\Sigma_n| & = 2|R|-|R\cap C|\\
               & = 2|R|-\left[\left(\sum_{i=1}^n|R_i\cap C|\right)-\left(\sum_{i=1}^{\lfloor n/2\rfloor}|R_i\cap R_{i'}\cap C|\right)\right]\\
               & = 2|R|-\Bigg[\left(M+2\sum_{i=1}^{\lfloor n/2\rfloor} 2|R_i\cap C_i|-|R_i\cap C_i\cap C_{i'}|\right)\\
               & \qquad -\left(\sum_{i=1}^{\lfloor n/2\rfloor}|R_i\cap R_{i'}\cap C|\right)\Bigg]\\
               & = 2|R|-\Bigg[\left(M+2\sum_{i=1}^{\lfloor n/2\rfloor} 2|R_i\cap C_i|-|R_i\cap C_i\cap C_{i'}|\right)\\
               & \qquad -\left(\sum_{i=1}^{\lfloor n/2\rfloor}2|R_i\cap R_{i'}\cap C_i|-|R_i\cap R_{i'}\cap C_i\cap C_{i'}|\right)\Bigg].
\end{split}
\]
From lemmas \ref{R} to \ref{RRCC}, this can be rewritten as 
\begin{equation} \label{maineq}
    \begin{split}
    |\Sigma_n|=&\left(\sum_{i=1}^{n}(-2)^{i+1}\binom{n}{i}\dfrac{(n^2-in)!}{(n-i)!^{n}}\right)-2\Bigg(\dfrac{(2\cdot\mathbf{1}_{2\mathbf{N}}(n+1)+4\lfloor n/2\rfloor)(n^2-2n+1)!}{(n-1)!(n-2)!^{n-1}}\\
    &\qquad-\dfrac{4\lfloor n/2\rfloor(n^2-3n+2)!}{(n-2)!^2(n-3)!^{n-2}}+\dfrac{\lfloor n/2\rfloor(n^2-4n+4)!}{(n-2)!^2(n-4)!^{n-2}}\Bigg)\\
    &=\left(\sum_{i=1}^{n}(-2)^{i+1}\binom{n}{i}\dfrac{(n^2-in)!}{(n-i)!^{n}}\right)-\dfrac{(4n)(n^2-2n+1)!}{(n-1)!(n-2)!^{n-1}}\\
    &\qquad+\dfrac{8\lfloor n/2\rfloor(n^2-3n+2)!}{(n-2)!^2(n-3)!^{n-2}}-\dfrac{2\lfloor n/2\rfloor(n^2-4n+4)!}{(n-2)!^2(n-4)!^{n-2}}.
\end{split}
\end{equation}

Thus we can calculate $P(\omega\in\Sigma_n)$ for several small values of $n$: 
\begin{table}[H]
    \centering
    \begin{tabular}{c|c}
        $n$ & $P(\omega\in\Sigma_n)$\\\hline
        2 & 1\\
        3 & 0.49047619\\
        4 & 0.090005867\\
        5 & 0.0097599064\\
        6 & $7.981508\cdot10^{-4}$\\
        7 & $5.326071\cdot10^{-5}$\\
        8 & $3.0083\cdot10^{-6}$
    \end{tabular}
    \label{table:1}
    \caption{CE$n$S probability across small dimensions.}
\end{table}
\end{proof}

\section{Proof of Corollary \ref{asymp_coro}}

\begin{proof}

The explicit formula for the size of $\Sigma_n$ in Theorem \ref{main_theorem} may be considered by some to be unwieldy or inelegant. Therefore, it behooves us to determine an asymptotic approximation of $|\Sigma_n|$. The growth of $|\Sigma_n|$ is dominated by the first term in Equation \eqref{maineq}, which is $2|R|$. Examining the sum in Lemma \ref{R}, we can similarly ignore all terms except for the first: 
\[|R|\sim (-1)^{1+1}\cdot 2^1\cdot\dfrac{n!(n^2-1n)!}{(n-1)!^{n+1}1!}.\]
Simplifying and multiplying by 2 yields our final asymptotic approximation of $|\Sigma_n|$, namely
\begin{equation}\label{approx}
    |\Sigma_n|\sim S(n):=\frac{4n\left(n^{2}-n\right)!}{\left(n-1\right)!^{n}}.
\end{equation}

The following table demonstrates the speed with which the ratio $|\Sigma_n|/S(n)$ approaches 1:
\begin{table}[h!]
    \centering
    \begin{tabular}{c|c|c|c}
        $n$ & $|\Sigma_n|$ & $S(n)$ & $|\Sigma_n|/S(n)$\\\hline
        2 & 6 & 16 & 0.375 \\
        3 & 824 & 1080 & 0.762 \\
        4 & $5.68\cdot10^6$ & $5.91\cdot10^6$ & 0.959 \\
        5 & $6.08\cdot10^{12}$ & $6.11\cdot10^{12}$ & 0.995 \\
        6 & $2.131\cdot10^{21}$ & $2.132\cdot10^{21}$ & 0.9996 \\
        7 & $3.9219\cdot10^{32}$ & $3.9220\cdot10^{32}$ & 0.99997\\
        8 & $5.46479\cdot10^{46}$ & $5.46480\cdot10^{46}$ & 0.999998 \\
        9 & $7.82733\cdot10^{63}$ & $7.82733\cdot10^{63}$ & 0.99999993\\
    \end{tabular}
    \label{table:2}
    \caption{Values of equations \eqref{maineq} and \eqref{approx} for small $n$}
\end{table}\\

Combining Lemma \ref{O} and Equation \eqref{approx}, we can also derive an approximation for $P(\omega\in\Sigma_n)$,
\[
    P(\omega\in\Sigma_n)=\dfrac{|\Sigma_n|}{|\Omega_n|}\sim\dfrac{4n\left(n^{2}-n\right)!}{\left(n-1\right)!^{n}}\cdot\left(\dfrac{n^2!}{n!^n}\right)^{-1}=\dfrac{4n^{n+1}(n^2-n)!}{n^2!}.
\]
Note that this approximation suffices to prove the true rarity of consecutive equi-$n$-squares among equi-$n$-squares, as 
\[\lim_{n\to\infty}\dfrac{4n^{n+1}(n^2-n)!}{n^2!}\leq \lim_{n\to\infty}\dfrac{4n^{n+1}}{(n^2-n+1)^n}\leq \lim_{n\to\infty}\dfrac{4n}{(n-1)^n}=0.\]
\end{proof}

\begin{remark} (An Elementary Bound for the Number of Latin Squares)

Let $Y_n:\Omega_n\to\mathbf{R}$ be a random variable counting the number of rows or columns in an E$n$S which are permutations. Note that for any Latin square $\lambda$ of order $n$, \[Y_n(\lambda)=2n.\] The number of E$n$Ss with a given row or column as a permutation is 
\[
n!\cdot\dfrac{(n^2-n)!}{(n-1)!^n}.
\]
Therefore, by Lemma \ref{O}, the first moment of $Y_n$ is \[\mathbf{E}(Y_n)=2n\cdot n!\cdot\dfrac{(n^2-n)!}{(n-1)!^n}\cdot\dfrac{n!^n}{n^2!}.\] By Markov's inequality, the probability that a randomly chosen E$n$S is a Latin square is bounded above as \[P(Y_n\geq 2n)=P(Y_n=2n)\leq n!\cdot\dfrac{(n^2-n)!}{(n-1)!^n}\cdot\dfrac{n!^n}{n^2!}.\]

Note that this bound on the probability could be used to bound the number of Latin squares of order $n$, $L_n$, as by Lemma \ref{O}, \[\dfrac{L_n}{|\Omega_n|}\leq n!\cdot\dfrac{(n^2-n)!}{(n-1)!^n}\cdot\dfrac{n!^n}{n^2!}\implies L_n\leq n!\cdot\dfrac{(n^2-n)!}{(n-1)!^n}.\] However, this bound is considerably weaker than the classical result: \[L_n\leq \prod_{k=1}^{n}(k!)^{n/k},\] as related by van Lint and Wilson in \cite{courseincombo} and is therefore of limited utility apart from the ease of its derivation.
\end{remark}

\section{Proof of Theorem \ref{distro_theorem}} \label{distro}

Consider the random variable $X_n:\Omega_n\to \mathbf{R}$ which counts the number of consecutive rows or columns in a random E$n$S $\omega$. From Theorem \ref{main_theorem}, observe that the number of E$n$Ss grows much faster than the number of CE$n$Ss---as exemplified by the values we calculated in Table \ref{table:1}. Therefore, as $n\to\infty$, the probability mass function of $X_n$ will approach the (discrete analog of the) Dirac distribution. However, for small values of $n$, the PMF of $X_n$ resists such a trivial characterization.

Let $x$ be the number of consecutive rows and columns in an E$n$S and consider the case when $x=0$. Applying Theorem \ref{main_theorem}, we can determine that the probability of generating an E$n$S that is not a CE$n$S:
\begin{equation}\label{pmf0}
\begin{split}
    P(X_n=0)&=|\Omega_n|^{-1}\cdot |\Omega_n\setminus\Sigma_n|=1-|\Sigma_n|\cdot|\Omega_n|^{-1}\\&
    \sim 1-\dfrac{4n^{n+1}(n^2-n)!}{n^2!}.
\end{split}
\end{equation}

The cases when $x$ is sufficiently large also prove rather simple. Recall that when $x>4$, there cannot be both consecutive rows and consecutive columns. Therefore, we need only consider the cases where there are either $x$ consecutive columns or $x$ consecutive rows. Such a distinction can be thought of constructively as a binary choice. Subsequently, there are $n$ possible choices for the location of our consecutive columns or rows. All that remains is to determine the position of the remaining $n-x$ copies of the elements. However, care must be taken to avoid placing these elements in a manner which would result in any additional consecutive rows or columns. Proceeding by the Principle of Inclusion and Exclusion, we consider the number of ways to fill the remaining rows, subtract those ways which would result in another given consecutive row, etc. Taking all of this into account, we can define our PMF at $x>4$ as
\begin{equation}\label{pmfx>4}
\begin{split}
    |\Omega_n|P(X_n=x) &= 2^{x+1}\binom{n}{x}\Bigg(\prod_{i=0}^{n-1}\binom{n^2-nx-i(n-x)}{n-x}\\
    &\qquad-\sum_{k=1}^{n-x}(-1)^{k+1}2^k\binom{n-x}{k}\dfrac{(n^2-nx-nk)!}{(n-x-k)!^n}\Bigg)\\
    &=2^{x+1}\binom{n}{x}\Bigg(\dfrac{(n^2-nx)!}{(n-x)!^n}\\
    &\qquad+\sum_{k=1}^{n-x}(-2)^k\binom{n-x}{k}\dfrac{(n^2-nx-nk)!}{(n-x-k)!^n}\Bigg)\\
    \implies P(X_n = x) &= \dfrac{n!^n}{n^2!}2^{x+1}\binom{n}{x}\sum_{k=0}^{n-x}(-2)^k\binom{n-x}{k}\dfrac{(n^2-nx-nk)!}{(n-x-k)!^n}.
\end{split}
\end{equation}

When $x=4$, we follow a similar procedure. Supposing that our four consecutive rows or columns are all rows or all columns, we have
\begin{equation} \label{pmf4 straighaway}
\begin{split}
    &2^5\binom{n}{4}\left(\prod_{i=0}^{n-1}\binom{n^2-4n-i(n-4)}{n-4}+\sum_{k=1}^{n-4}(-2)^k\binom{n-4}{k}\dfrac{(n^2-4n-nk)!}{(n-k-4)!^n}\right)\\
    =&2^5\binom{n}{4}\sum_{k=0}^{n-4}(-2)^k\binom{n-4}{k}\dfrac{(n^2-4n-nk)!}{(n-k-4)!^n}
\end{split}
\end{equation}
possible arrangements. However, it is possible to have four consecutive rows or columns that are not all rows or all columns. If we pick row $i$ to be consecutive, then row $i'$, column $i$, and column $i'$ may also be consecutive. Note that if $n$ is odd, our initial choice of row $i$ is restricted to the $n-1$ rows that are not the middle row. Where these consecutive rows and columns intersect, there will be elements that only appear twice, whereas all other elements have already appeared four times. So, being careful to account for this, we complete our construction by placing these elements appearing twice, and the rest. This is similar to the content of Lemma \ref{RRCC}. Therefore, the total number of ways that such E$n$S may be constructed is 
\begin{equation} \label{pmf4 crisscross}
\begin{split}
    &(n+\mathbf{1}_{2\mathbf{N}}(n)-1)\binom{n^2-4n+4}{n-2}\binom{n^2-4n+4-(n-2)}{n-2}\\
    &\qquad\cdot\prod_{j=0}^{n-3}\binom{n^2-4n+4-2(n-2)-j(n-4)}{n-4}\\
    =&\dfrac{(n+\mathbf{1}_{2\mathbf{N}}(n)-1)(n^2-4n+4)!}{(n-2)!^2(n-4)!^{n-2}}.
\end{split}
\end{equation}

The value of the PMF at $x=4$ is simply the sum of expressions \eqref{pmf4 straighaway} and \eqref{pmf4 crisscross} divided by the total number of E$n$Ss:
\begin{equation}\label{pmf4}
\begin{split}
    P(X_n=4)=&\dfrac{n!^n}{n^2!}\Bigg(2^5\binom{n}{4}\left(\sum_{k=0}^{n-4}(-2)^k\binom{n-4}{k}\dfrac{(n^2-4n-nk)!}{(n-k-4)!^n}\right)\\
    &\qquad+\dfrac{(n+\mathbf{1}_{2\mathbf{N}}(n)-1)(n^2-4n+4)!}{(n-2)!^2(n-4)!^{n-2}}\Bigg).
\end{split}
\end{equation}

When $x=3$, there are two possible patterns the E$n$S may follow: either there are two consecutive rows and one consecutive column (or one consecutive row and two consecutive columns) or there are three consecutive rows/columns. 

The first case is the more delicate. First, we must choose whether there will be two consecutive rows and one consecutive column or one consecutive row and two consecutive columns. Assume, without loss of generality, that we decide that there should be two consecutive rows. After this, we must decide the location of one of these rows, say row $i$, out of the $n+\mathbf{1}_{2\mathbf{N}}(n)-1$ available rows. We can assume row $i$ is not reverse-consecutive, because---if there is to be a consecutive column---the other consecutive row must be located at row $i'$ and be in reverse-consecutive order. There are now two choices for the location of the consecutive column, column $i$ or column $i'$. Suppose we choose column $i$. Once that choice has been made, we fill the array with the remaining elements. Some care must be taken because elements $i$ and $i'$ have only been used twice, whereas all other elements have been used three times. However, further attention must be given to the possibility that---in placing the remaining elements---column $i'$ is filled consecutively. If this were to occur, we would have two consecutive columns and two consecutive rows. In order to avoid counting these arrangements, we may merely subtract by expression \eqref{pmf4 crisscross} four times. We subtract four times because each E$n$S with two consecutive rows and two consecutive columns is over-counted four times. Take for example the following E$4$S, which has two consecutive rows and two consecutive columns:

\[\begin{array}{|c|c|c|c|}\hline
    1 & 2 & 3 & 4\\\hline
    2 & \cdot & \cdot & 3\\\hline
    3 & \cdot & \cdot & 2\\\hline
    4 & 3 & 2 & 1\\\hline
\end{array}.\]

This E$4$S corresponds to 4 E$4$Ss with either two consecutive rows and one consecutive column or one consecutive row and two consecutive columns

\[\begin{array}{|c|c|c|c|}\hline
    1 & 2 & 3 & 4\\\hline
    \cdot & \cdot & \cdot & 3\\\hline
    \cdot & \cdot & \cdot & 2\\\hline
    4 & 3 & 2 & 1\\\hline
\end{array},\;\begin{array}{|c|c|c|c|}\hline
    1 & \cdot & \cdot & 4\\\hline
    2 & \cdot & \cdot & 3\\\hline
    3 & \cdot & \cdot & 2\\\hline
    4 & 3 & 2 & 1\\\hline
\end{array},\;\begin{array}{|c|c|c|c|}\hline
    1 & 2 & 3 & 4\\\hline
    2 & \cdot & \cdot & \cdot\\\hline
    3 & \cdot & \cdot & \cdot\\\hline
    4 & 3 & 2 & 1\\\hline
\end{array},\;\text{and}\;\begin{array}{|c|c|c|c|}\hline
    1 & 2 & 3 & 4\\\hline
    2 & \cdot & \cdot & 3\\\hline
    3 & \cdot & \cdot & 2\\\hline
    4 & \cdot & \cdot & 1\\\hline
\end{array}.\]

Following this procedure, we have
\begin{equation} \label{pmf3 crisscross}
    \begin{split}
        &4(n+\mathbf{1}_{2\mathbf{N}}(n)-1)\binom{n^2-3n+2}{n-2}\binom{n^2-3n+2-(n-2)}{n-2}\\
        &\qquad\cdot\left(\prod_{i=0}^{n-3}\binom{n^2+2n-2(n-2)-i(n-3)}{n-3}\right)-\dfrac{4(n+\mathbf{1}_{2\mathbf{N}}(n)-1)(n^2-4n+4)!}{(n-2)!^2(n-4)!^{n-2}}\\
        =&\dfrac{4(n+\mathbf{1}_{2\mathbf{N}}(n)-1)(n^2-3n+2)!}{(n-2)!^2(n-3)!^{n-2}}-\dfrac{4(n+\mathbf{1}_{2\mathbf{N}}(n)-1)(n^2-4n+4)!}{(n-2)!^2(n-4)!^{n-2}}
    \end{split}
\end{equation}
total arrangements.

The case when there are three consecutive rows or three consecutive columns is simpler, and is directly analogous to the derivation of expression \eqref{pmf4 straighaway}. The number of such E$n$Ss is 
\begin{equation} \label{pmf3 straighaway}
        2^4\binom{n}{3}\left(\sum_{k=0}^{n-3}(-2)^k\binom{n-3}{k}\dfrac{(n^2-3n-nk)!}{(n-k-3)!^n}\right).
\end{equation}

The value of the PMF at $x=3$ is simply the sum of expressions \eqref{pmf3 crisscross} and \eqref{pmf3 straighaway} divided by the total number of E$n$S:
\begin{equation}\label{pmf3}
\begin{split}
    P(X_n=3)=&\dfrac{n!^n}{n^2!}\Bigg(2^4\binom{n}{3}\left(\sum_{k=0}^{n-3}(-2)^k\binom{n-3}{k}\dfrac{(n^2-3n-nk)!}{(n-k-3)!^n}\right)\\
    &\qquad+\dfrac{4(n+\mathbf{1}_{2\mathbf{N}}(n)-1)(n^2-3n+2)!}{(n-2)!^2(n-3)!^{n-2}}\\
    &\qquad-\dfrac{4(n+\mathbf{1}_{2\mathbf{N}}(n)-1)(n^2-4n+4)!}{(n-2)!^2(n-4)!^{n-2}}\Bigg).
\end{split}
\end{equation}

The case when $x=2$ is similar. Either there is one consecutive row and one consecutive column or there are two consecutive rows or columns.

The first case, again, is more delicate. If $n$ is even, it is clear that there are $n$ choices for the location for the consecutive row $i$. Then, one may choose among columns $i$ or $i'$ to be consecutive. If $n$ is odd, some special attention must be given to the possibility that the row chosen to be consecutive is row $(n+1)/2$. However, if said row is chosen, the fact that the choice of consecutive column is fixed is counteracted by the fact that said column may be consecutive or reverse consecutive. Once the consecutive row and column are fixed, we place the remaining elements.

The second case, where there are two consecutive rows or columns, necessitates similar arguments to those above. Constructing an arbitrary element requires first determining if rows or columns will be consecutive, then what order they will be in. Finally, we choose which rows or columns will be consecutive and permute the remaining elements. Being careful to avoid creating anymore consecutive columns or rows.

To avoid counting the E$n$Ss with 3 or more consecutive rows or columns, we subtract by the quantities in expression \eqref{pmf3 crisscross} three times.  We subtract three times because each E$n$S with two consecutive rows and one consecutive column or one consecutive row and two consecutive columns is over-counted thrice. Consider the following diagram:

\footnotesize
\renewcommand*{\arraystretch}{1.4375}
\begin{center}\hspace*{-0.75cm}
    \begin{tikzpicture}[scale=0.75]
        \draw
        (-7.5,0) node (a1) {$\begin{array}{|c|c|c|c|}\hline
            1 & 2 & 3 & 4 \\\hline
            \hspace{1ex} & \hspace{1ex} & \hspace{1ex} & 3 \\\hline
            \hspace{1ex} & \hspace{1ex} & \hspace{1ex} & 2 \\\hline
            \hspace{1ex} & \hspace{1ex} & \hspace{1ex} & 1 \\\hline
        \end{array}$}
        (-4.5,0) node (a2) {$\begin{array}{|c|c|c|c|}\hline
            1 & 2 & 3 & 4 \\\hline
            2 & \hspace{1ex} & \hspace{1ex} & \hspace{1ex} \\\hline
            3 & \hspace{1ex} & \hspace{1ex} & \hspace{1ex} \\\hline
            4 & \hspace{1ex} & \hspace{1ex} & \hspace{1ex} \\\hline
        \end{array}$}
        (-1.5,0) node (a3) {$\begin{array}{|c|c|c|c|}\hline
            1 & \hspace{1ex} & \hspace{1ex} & \hspace{1ex} \\\hline
            2 & \hspace{1ex} & \hspace{1ex} & \hspace{1ex} \\\hline
            3 & \hspace{1ex} & \hspace{1ex} & \hspace{1ex} \\\hline
            4 & 3 & 2 & 1 \\\hline
        \end{array}$}
        (1.5,0) node (a4) {$\begin{array}{|c|c|c|c|}\hline
            \hspace{1ex} & \hspace{1ex} & \hspace{1ex} & 4 \\\hline
            \hspace{1ex} & \hspace{1ex} & \hspace{1ex} & 3 \\\hline
            \hspace{1ex} & \hspace{1ex} & \hspace{1ex} & 2 \\\hline
            4 & 3 & 2 & 1 \\\hline
        \end{array}$}
        (4.5,0) node (a5) {$\begin{array}{|c|c|c|c|}\hline
            1 & \hspace{1ex} & \hspace{1ex} & 4 \\\hline
            2 & \hspace{1ex} & \hspace{1ex} & 3 \\\hline
            3 & \hspace{1ex} & \hspace{1ex} & 2 \\\hline
            4 & \hspace{1ex} & \hspace{1ex} & 1 \\\hline
        \end{array}$}
        (7.5,0) node (a6) {$\begin{array}{|c|c|c|c|}\hline
            1 & 2 & 3 & 4 \\\hline
            \hspace{1ex} & \hspace{1ex} & \hspace{1ex} & \hspace{1ex} \\\hline
            \hspace{1ex} & \hspace{1ex} & \hspace{1ex} & \hspace{1ex} \\\hline
            4 & 3 & 2 & 1 \\\hline
        \end{array}$}
        (-4.5,-4.5) node (a7) {$\begin{array}{|c|c|c|c|}\hline
            1 & 2 & 3 & 4 \\\hline
            \hspace{1ex} & \hspace{1ex} & \hspace{1ex} & 3 \\\hline
            \hspace{1ex} & \hspace{1ex} & \hspace{1ex} & 2 \\\hline
            4 & 3 & 2 & 1 \\\hline
        \end{array}$}
        (-1.5,-4.5) node (a8) {$\begin{array}{|c|c|c|c|}\hline
            1 & 2 & 3 & 4 \\\hline
            2 & \hspace{1ex} & \hspace{1ex} & 3 \\\hline
            3 & \hspace{1ex} & \hspace{1ex} & 2 \\\hline
            4 &  &  & 1 \\\hline
        \end{array}$}
        (1.5,-4.5) node (a9) {$\begin{array}{|c|c|c|c|}\hline
            1 & 2 & 3 & 4 \\\hline
            2 & \hspace{1ex} & \hspace{1ex} &  \\\hline
            3 & \hspace{1ex} & \hspace{1ex} &  \\\hline
            4 & 3 & 2 & 1 \\\hline
        \end{array}$}
        (4.5,-4.5) node (a10) {$\begin{array}{|c|c|c|c|}\hline
            1 &  &  & 4 \\\hline
            2 & \hspace{1ex} & \hspace{1ex} & 2 \\\hline
            3 & \hspace{1ex} & \hspace{1ex} & 3 \\\hline
            4 & 3 & 2 & 1 \\\hline
        \end{array}$}
        (0,-9) node (a11) {$\begin{array}{|c|c|c|c|}\hline
            1 & 2 & 3 & 4 \\\hline
            2 & \hspace{1ex} & \hspace{1ex} & 2 \\\hline
            3 & \hspace{1ex} & \hspace{1ex} & 3 \\\hline
            4 & 3 & 2 & 1 \\\hline
        \end{array}$};

        \draw (-7.5,-1.5) -- (-4.5, -3);
        \draw (-7.5,-1.5) -- (-1.5, -3);
        \draw (-4.5,-1.5) -- (-1.5,-3);
        \draw (-4.5,-1.5) -- (1.5,-3);
        \draw (-1.5,-1.5) -- (1.5,-3);
        \draw (-1.5,-1.5) -- (4.5,-3);
        \draw (1.5,-1.5) -- (-4.5, -3);
        \draw (1.5,-1.5) -- (4.5,-3);
        \draw (4.5,-1.5) -- (-1.5,-3);
        \draw (4.5,-1.5) -- (4.5,-3);
        \draw (7.5,-1.5) -- (-4.5,-3);
        \draw (7.5,-1.5) -- (1.5,-3);
        \draw (-4.5,-6) -- (0,-7.5);
        \draw (-1.5,-6) -- (0,-7.5);
        \draw (1.5,-6) -- (0,-7.5);
        \draw (4.5,-6) -- (0,-7.5);
\end{tikzpicture}
\end{center}

\normalsize
\newpage

Thus, the number of E$n$Ss with two consecutive rows or columns is

\begin{equation}\label{pmf2 crisscross and straightaway}
    \begin{split}
        &\dfrac{4n(n^2-2n+1)!}{(n-1)!(n-2)!^{n-1}}+2^3\binom{n}{2}\left(\sum_{k=0}^{n-2}(-2)^k\binom{n-2}{k}\dfrac{(n^2-2n-nk)!}{(n-k-2)!^n}\right)
        \\&-3\bigg(\dfrac{4(n+\mathbf{1}_{2\mathbf{N}}(n)-1)(n^2-3n+2)!}{(n-2)!^2(n-3)!^{n-2}}-\dfrac{4(n+\mathbf{1}_{2\mathbf{N}}(n)-1)(n^2-4n+4)!}{(n-2)!^2(n-4)!^{n-2}}\bigg).
    \end{split}
\end{equation}

So, the value of the PMF at $x=2$ is simply \eqref{pmf2 crisscross and straightaway} divided by the total number of E$n$Ss:

\begin{equation}\label{pmf2}
    \begin{split}
        P(X_n=2)=&\dfrac{n!^n}{n^2!}\Bigg(\dfrac{4n(n^2-2n+1)!}{(n-1)!(n-2)!^{n-1}}+2^3\binom{n}{2}\left(\sum_{k=0}^{n-2}(-2)^k\binom{n-2}{k}\dfrac{(n^2-2n-nk)!}{(n-k-2)!^n}\right)\\
        &\qquad-3\bigg(\dfrac{4(n+\mathbf{1}_{2\mathbf{N}}(n)-1)(n^2-3n+2)!}{(n-2)!^2(n-3)!^{n-2}}\\
        &\qquad-\dfrac{4(n+\mathbf{1}_{2\mathbf{N}}(n)-1)(n^2-4n+4)!}{(n-2)!^2(n-4)!^{n-2}}\bigg)\Bigg).
    \end{split}
\end{equation}

All that remains to be done is to determine the value of the PMF at $x=1$. This can be done crudely and without too much of a sophisticated argument, subtracting the number of E$n$Ss with more than one consecutive row or column from the number of E$n$Ss with at least one consecutive row or column, $|\Sigma_n|$. Doing so amounts to combining the main result of theorem \ref{main_theorem}, equation \eqref{maineq}, with equations \eqref{pmfx>4}, \eqref{pmf4}, \eqref{pmf3}, and \eqref{pmf2} to get

\begin{equation}
    \begin{split}
        P(X_n=1)=\dfrac{n!^n}{n^2!}&\Bigg(\left(n\sum_{i=1}^{n}(-2)^{i+1}\binom{n-1}{i-1}\dfrac{(n^2-in)!}{(n-i)!^n}\right)\\
        &-\dfrac{(8n)(n^2-2n+1)!}{(n-1)!(n-2)!^{n-1}}\\
        &+\dfrac{8(\mathbf{1}_{2\mathbf{N}}(n)+n+\lfloor n/2\rfloor-1)(n^2-3n+2)!}{(n-2)!^2(n-3)!^{n-2}}\\
        &-\dfrac{(9\cdot\mathbf{1}_{2\mathbf{N}}(n)+9n+2\lfloor n/2\rfloor-9)(n^2-4n+4)!}{(n-2)!^2(n-4)!^{n-2}}\Bigg).
    \end{split}
\end{equation}

Combining the above with equations \eqref{pmf0}, \eqref{pmfx>4}, \eqref{pmf4}, \eqref{pmf3}, and \eqref{pmf2} yields the PMF stated in theorem \ref{distro_theorem}.\qed

\section{Remark on the Algebra of Consecutive Equi-$n$-Squares} \label{algebraic interpretations}
It is well known that Latin squares can be interpreted as the Cayley tables of quasigroups \cite{SmallLSsQGsLs}. Given that E$n$Ss are less structured, it is intuitive that their algebraic interpretation would be an object with fewer axioms. In fact, E$n$Ss are the Cayley tables of certain magmas. (Note that not all magmas correspond to an E$n$S, however.) CE$n$Ss correspond to magmas which are nearly unital. 

Precisely, if row $i$ is consecutive, then left multiplication by $i$ is the identity map. If column $i$ is consecutive, then right multiplication by $i$ is the identity map. If row or column $i$ is reverse-consecutive, then right or left multiplication (respectively) with $i$ is the identity map if one takes the ``reverse'' of the image. More exactly, consider a magma $(M,\circ)$ with its Cayley table a CE$n$S. Let row $i$ be reverse-consecutive and define the map
\[
\begin{split}
    I:&M\to M\\
    &m\mapsto i\circ m
\end{split}
\]
where $m$ is any element of $M$. Let $R$ be the reversing map upon the elements of $M$ sending the $n^{\text{th}}$ element of $M$ to the $(|M|-n+1)^{\text{th}}$ element of $M$. The composition $R\circ I$ is the identity map $\iota:M\to M$. A similar argument explains the case when the CE$n$S has a reverse-consecutive column. As far as we are aware, this property of being guaranteed (an element akin to) a left or right identity, is yet to have a common name in the relevant literature. Again, not every magma with this property would have a CE$n$S as its Cayley table. However, the following Theorems \ref{no QGs are Ls} and \ref{completely simple} do provide some concrete connections between equi-$n$-squares and finite algebra.

\section{Proof of Theorem \ref{no QGs are Ls}}

\begin{proof}
    For a given row of a randomly generated Latin square, there are $n!$ possible arrangements. Of these $n!$ arrangements, only one is consecutive but not reverse-consecutive. The chance, therefore, that any row is consecutive is $1/n!$. By the linearity of the expectation, the expected number of consecutive or reverse-consecutive rows or columns in a randomly generated order $n$ Latin square is $2n/n!$. 
    
    By Markov's inequality, the probability that there exists at least two such rows or columns in a randomly generated Latin square is less than $n/n!$. Of course, as $n\to\infty$ this probability swiftly approaches $0$, implying almost no quasigroups are loops, confirming the pattern apparent in Table 3 of \cite{SmallLSsQGsLs}.
\end{proof}

\section{Proof of Theorem \ref{completely simple}}
\begin{proof}
    Suppose that $S$ is a finite (completely) simple semigroup. By Rees' theorem for completely simple semigroups, \[S\cong M(G;I,\Lambda;P),\] for some index sets $I,\Lambda$, group $G$, and matrix $P$ with non-zero entries from $G$ \cite{HowieSemigroups}. Elements of $M(G;I,\Lambda;P)$ are triplets $(i,g,\lambda)$ with $(i,g,\lambda)\in I\times G\times \Lambda$ and the operation on $M(G;I,\Lambda;P)$ is defined by \[(i,g,\lambda)(j,t,\mu)=(i,gp_{\lambda,j}h,\mu),\] where $p_{\lambda,\mu}\in G$.
    
    Define $\mu$ to be the multiplication map on  $\mu:S\times S\to S$. It suffices to show that every fiber of $\mu$ has size $n$. Consider an arbitrary element $(i_0,g_0,\lambda_0)$ of $M(G;I,\Lambda;P)$. Any preimage of $(i_0,g_0,\lambda_0)$ must be of the form \[(i_0,g_1,\lambda_1)(j_1,h_1,\lambda_0).\] We may therefore choose any $j_1\in I$ and $\lambda_1\in \Lambda$. Finally, as $p_{\lambda_1j_1}\neq 0$ is non-zero, any choice of $g_1\in G$ fixes $h_1$ to be $g_1^{-1}p_{\lambda_1j_1}^{-1}g_0$. Therefore, by isomorphism, any fiber of $\mu$ is of size \[|I|\cdot|G|\cdot|\Lambda|=|I\times G\times \Lambda|=|S|=n.\]
\end{proof}

\section{Proof of Corollary \ref{nonLatin_coro}}

It may not be immediately obvious that an equi-$n$-square that can be interpreted as an associative Cayley table is not simply a Latin square, given the more complicated relationship between associativity and the structure of the Cayley table than unity or commutativity. However, the above proof shows that there are such highly ordered equi-$n$-squares that are not Latin. 

\begin{proof}
    Consider a finite simple semigroup $S$ that is not also a group. By the above Theorem \ref{completely simple}, the Cayley table of $S$ is an associative equi-$n$-square. However, it cannot also be a Latin square, because then $S$ would also be a group.
\end{proof}

Enumerating such associative equi-$n$-squares remains an open problem.

\section{Code and Monte-Carlo Results} \label{code}
The following python code implements a Monte-Carlo simulation to estimate the probabilities in Theorem \ref{distro_theorem} for $n\in\{3,4,5\}$, generating $10^8$ uniformly random E$n$Ss for each value of $n$ and creating plots of the differences between the experimental and theoretical probabilities:

\footnotesize\begin{lstlisting}[language=Python,breaklines]
import random
import matplotlib.pyplot as plt
from tqdm import tqdm

def generate_elements(dimension):
  elements = []
  for i in range(dimension):
    [elements.append(i + 1) for _ in range(dimension)]
  return elements

def generate_matrix(elements, dimension):
  random.shuffle(elements)
  matrix = []
  for i in range(dimension):
    row = []
    for j in range(dimension):
      row.append(elements[i * dimension + j])
    matrix.append(row)
  return matrix

def check_row(matrix, dimension, reference):
  hits = 0
  locations = []
  for i in range(dimension):
    row = matrix[i]
    if (row == reference) or (row == reference[::-1]):
      hits += 1
      locations.append(i)
  return (hits, locations)

def check_column(matrix, dimension, reference):
  hits = 0
  locations = []
  for i in range(dimension):
    column = [matrix[j][i] for j in range(dimension)]
    if (column == reference) or (column == reference[::-1]):
      hits += 1
      locations.append(i)
  return (hits, locations)


def check_matrix(matrix, dimension, reference):
  row_hits = check_row(matrix, dimension, reference)[0]
  col_hits = check_column(matrix, dimension, reference)[0]
  return (row_hits, col_hits, row_hits + col_hits)

def main(dimension, iterations):
  elements = generate_elements(dimension)
  reference = [i + 1 for i in range(dimension)]
  distribution = [0]*(dimension + 1)
  running_probs = []
  for i in tqdm(range(iterations)):
    matrix = generate_matrix(elements, dimension)
    consec_count_tots = check_matrix(matrix, dimension, reference)[2]
    distribution[consec_count_tots] += 1
    running_probs.append([count/(i+1) for count in distribution])
  return(running_probs)


calculated_probs =
[[0.509523809524,0.357142857143,0.114285714286,0.0190476190476],
[0.909994132851,0.0863048063048,0.0036336996337,0.0000673612102184],
[0.990240093585,0.00971688511463,0.0000429497896608,0.000000071510797707]]

def makebounds(p):
  s = 5*((p-p**2)/iterations)**0.5
  bound = max(s, 0.01*p)
  return([-bound, bound])

dimension = 5
iterations = 100000000
resolution = 10000
running_probs = main(dimension, iterations)
unzipped_probs = []
for consecutive in range(len(running_probs[0])):
  unzipped_probs.append([probs[consecutive] for probs in running_probs])

small_sum_probs = [0]*iterations
for i in tqdm(range(len(unzipped_probs) - 3)):
  for j in range(iterations):
    small_sum_probs[j] += unzipped_probs[i + 3][j]

relative_probs = [[probs - calculated_probs[dimension - 3][consecutive] for probs in unzipped_probs[consecutive]] for consecutive in range(3)]

relative_probs.append([probs - calculated_probs[dimension - 3][3] for probs in small_sum_probs])

x_values = [int(resolution*x) for x in range(int(iterations/resolution))]
y_values = relative_probs
dummy_x = [0,iterations-resolution]

lines = []
fig, axs = plt.subplots(2, 2)
fig.tight_layout()

axs[0,0].plot(x_values, [y_values[0][y] for y in range(iterations) if (y % resolution == 0)])
axs[0,0].plot(dummy_x, [0, 0], color='black', linestyle='-')
axs[0,0].set_title(r'$x=0$')
axs[0,0].set_xlabel('Iterations')
axs[0,0].set_ylabel(r'Experimental Probability - $p_{X_'+ str(dimension) + r'}(0)$')
axs[0,0].set_ylim(makebounds(calculated_probs[dimension - 3][0]))

axs[0,1].plot(x_values, [y_values[1][y] for y in range(iterations) if (y % resolution == 0)], 'tab:orange')
axs[0,1].plot(dummy_x, [0, 0], color='black', linestyle='-')
axs[0,1].set_title(r'$x=1$')
axs[0,1].set_xlabel('Iterations')
axs[0,1].set_ylabel(r'Experimental Probability - $p_{X_'+ str(dimension) + r'}(1)$')
axs[0,1].set_ylim(makebounds(calculated_probs[dimension - 3][1]))

axs[1,0].plot(x_values, [y_values[2][y] for y in range(iterations) if (y % resolution == 0)], 'tab:green')
axs[1,0].plot(dummy_x, [0, 0], color='black', linestyle='-')
axs[1,0].set_title(r'$x=2$')
axs[1,0].set_xlabel('Iterations')
axs[1,0].set_ylabel(r'Experimental Probability - $p_{X_'+ str(dimension) + r'}(2)$')
axs[1,0].set_ylim(makebounds(calculated_probs[dimension - 3][2]))

axs[1,1].plot(x_values, [y_values[3][y] for y in range(iterations) if (y % resolution == 0)], 'tab:red')
axs[1,1].plot(dummy_x, [0, 0], color='black', linestyle='-')
axs[1,1].set_title(r'$x \geq 3$')
axs[1,1].set_xlabel('Iterations')
axs[1,1].set_ylabel(r'Experimental Probability - $\sum_{i=3}^{\infty} p_{X_' + str(dimension) + r'}(i)$')
axs[1,1].set_ylim(makebounds(calculated_probs[dimension - 3][3]))

plt.show()
\end{lstlisting}
\normalsize

The following are the plots produced by the above code for $n\in\{3,4,5\}$. The three plots in each graphic show the difference between experimental probability and the probability predicted by the result of Theorem \ref{distro_theorem} as the simulation progresses. Please take note that the code above scales the $y$-axis in each plot relative to the probability. Specifically, $y$ belongs to the interval centered at zero with width \[\max\left\{5\sqrt{\dfrac{p(1-p)}{N}},\dfrac{p}{100}\right\},\] where $p$ is the relevant theoretical probability and $N$ is the number of iterations in the simulation (i.e. the number of sampled E$n$Ss), for our purposes $N=10^8$.

\begin{figure}[h!]
    \centering
    \includegraphics[width=1\linewidth]{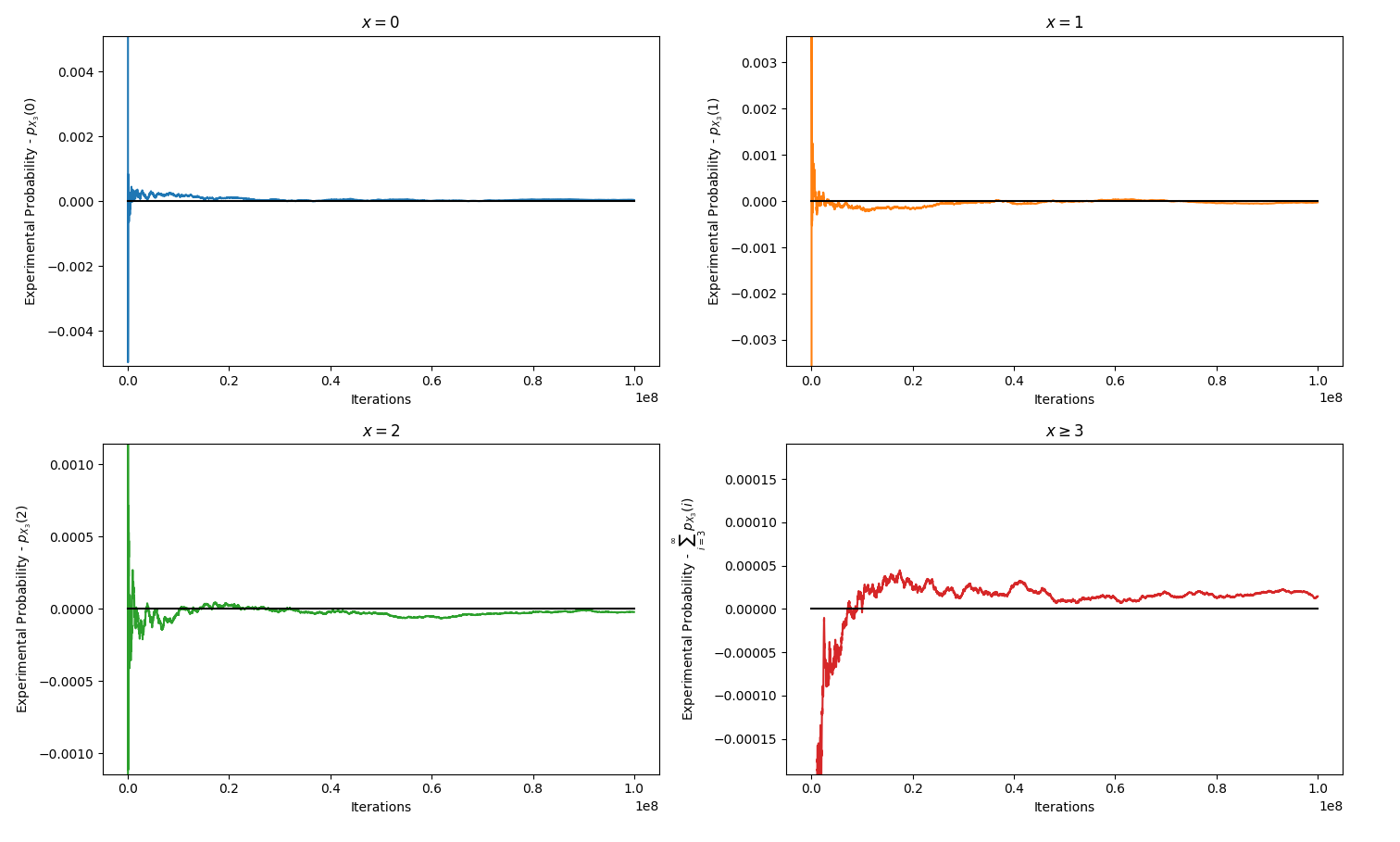}
    \caption{Differences between simulated and predicted probabilities for $p_{X_3}(x)$}
    \label{fig:mcplot3}
\end{figure}

\begin{figure}[h!]
    \centering
    \includegraphics[width=1\linewidth]{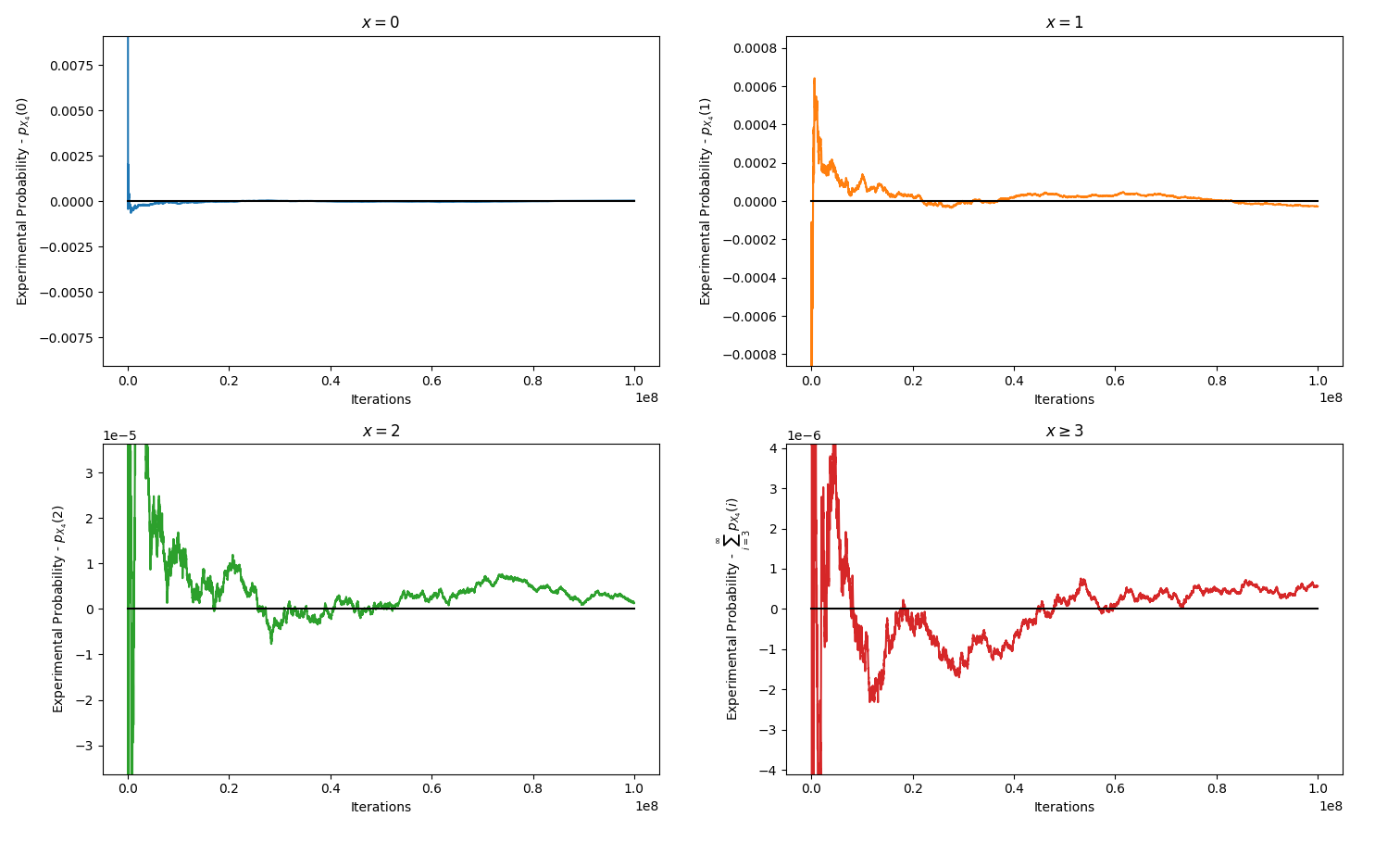}
    \caption{Differences between simulated and predicted probabilities for $p_{X_4}(x)$}
    \label{fig:mcplot4}
\end{figure}

\begin{figure}[h!]
    \centering
    \includegraphics[width=1\linewidth]{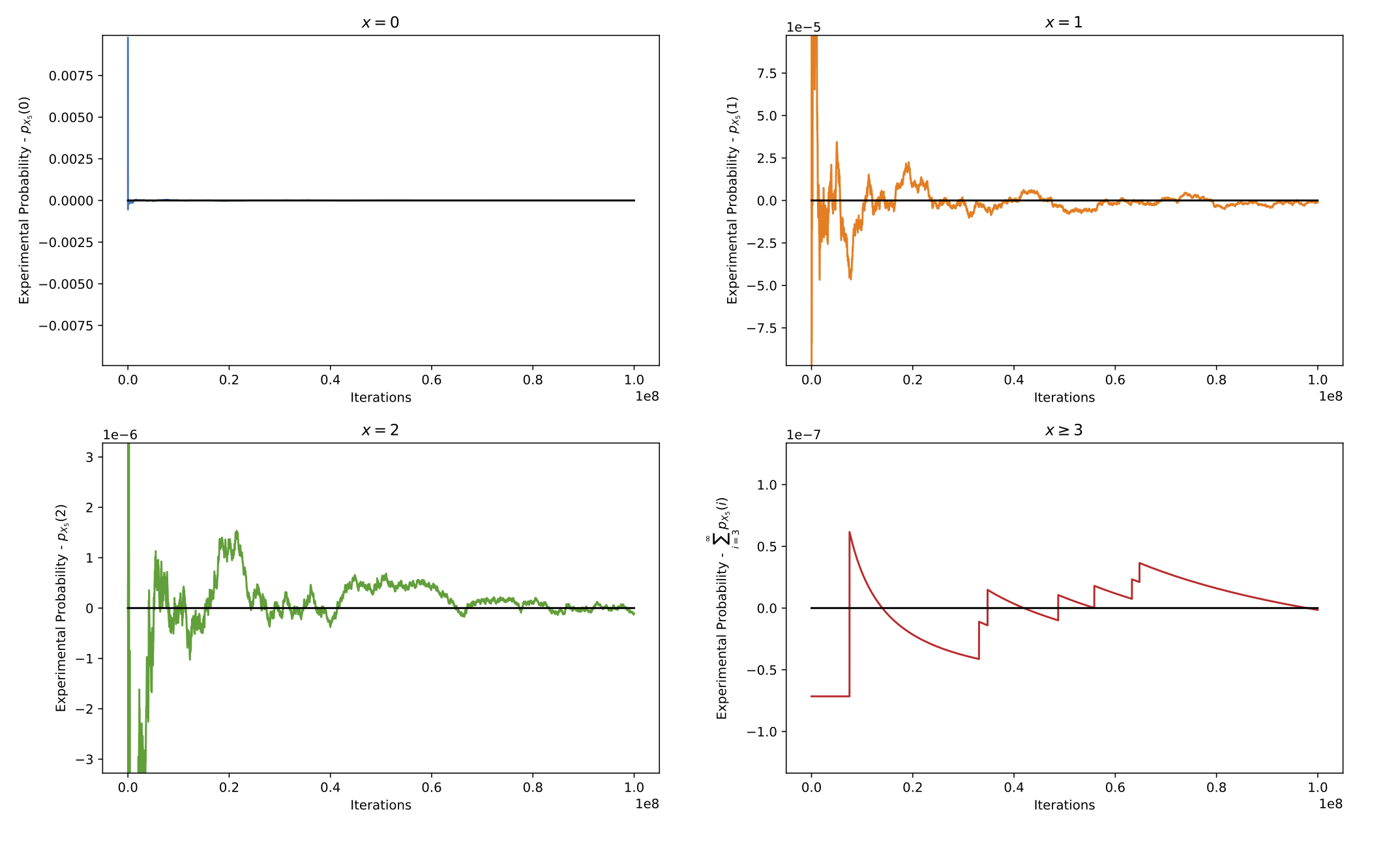}
    \caption{Differences between simulated and predicted probabilities for $p_{X_5}(x)$}
    \label{fig:mcplot5}
\end{figure}

\newpage

\section{Acknowledgments}
The author is grateful to Professor Fan Ge at the College of William \& Mary, Professor Nawaf Bou-Rabee at Rutgers University, and the mathematics department at the Maret School for their patient discussion and advice.

\bibliographystyle{plain}
\bibliography{pls_cite}

\end{document}